\documentclass[11pt]{amsart}
\usepackage[labelfont=bf]{caption}
\renewcommand{\figurename}{Fig.}

\usepackage{amsmath,amsfonts,amssymb,amsthm,booktabs,color,epsfig,graphicx,hyperref,url}

\usepackage{enumerate}
\usepackage{tikz-cd}
\usepackage{subfig}
\usepackage{float}	
\usepackage{xcolor}

\theoremstyle{plain}
\newtheorem{theorem}{Theorem}
\newtheorem{corollary}{Corollary}
\newtheorem{lemma}{Lemma}
\newtheorem{proposition}{Proposition}

\theoremstyle{definition}
\newtheorem{definition}{Definition}
\newtheorem*{remark}{Remark}
\newtheorem*{remarks}{Remarks}

\DeclareMathOperator{\spn}{span}
\DeclareMathOperator{\diag}{diag}
\DeclareMathOperator{\sym}{Sym}
\renewcommand{\hat}{\widehat}
\renewcommand{\tilde}{\widetilde}

\begin{document}

\title{A Cram\'er--Wold theorem for elliptical distributions}

\date{8 March 2023}

\author{Ricardo Fraiman}
\address{Centro de Matem\'atica, Facultad de Ciencias, Universidad de la Rep\'ublica, Uruguay.} 
\email{rfraiman@cmat.edu.uy} 

\author{Leonardo Moreno}
\address{Instituto de Estad\'{i}stica, Departamento de M\'etodos Cuantitativos, FCEA, Universidad de la Rep\'ublica, Uruguay.}
\email{mrleo@iesta.edu.uy} 

\author{Thomas Ransford}
\address{D\'epartement de math\'ematiques et de statistique, Universit\'e Laval,
Qu\'ebec City (Qu\'ebec),  Canada G1V 0A6.}
\email{ransford@mat.ulaval.ca}

\begin{abstract}
According to a well-known theorem of Cram\'er and Wold, 
if $P$ and $Q$ are two Borel probability measures on $\mathbb{R}^d$
whose projections $P_L,Q_L$ onto each line $L$ in $\mathbb{R}^d$ satisfy $P_L=Q_L$, then $P=Q$.
Our main result is that, if $P$ and $Q$ are both elliptical distributions,
then, to show that $P=Q$, it suffices merely to check that $P_L=Q_L$ for a certain set of $(d^2+d)/2$ lines $L$.
Moreover $(d^2+d)/2$ is optimal. The class of elliptical distributions contains the Gaussian
distributions as well as many other multivariate distributions of interest.
Our theorem contrasts with  other variants of the Cram\'er--Wold theorem, 
in that no assumption is made about the finiteness of moments of $P$ and $Q$.
We use our results to derive a statistical test for equality of elliptical distributions,
and carry out a small simulation study of the test, comparing it
with other tests from the literature. We also give an
application to learning (binary classification), again illustrated with a small simulation.
\end{abstract}

\keywords{Cram\'er--Wold, projection, elliptical distribution, Kolmogorov--Smirnov test}

\makeatletter
\@namedef{subjclassname@2020}{\textup{2020} Mathematics Subject Classification}
\makeatother

\subjclass[2020]{Primary 60B11, Secondary 60E10, 62H15}

\maketitle

\section{Introduction and statement of main results}

Given a Borel probability measure $P$ on $\mathbb{R}^d$ and a vector subspace $H$ of $\mathbb{R}^d$,
we write $P_H$ for the projection of $P$ onto $H$, namely the Borel probability measure on $H$ given by
\[
P_H(B):=P(\pi_H^{-1}(B)),
\]
where $\pi_H:\mathbb{R}^d\to H$ is the orthogonal projection of $\mathbb{R}^d$ onto $H$.
We shall be particularly interested in the case when $H$ is a line $L$.
One can  view $P_L$ as the marginal distribution of $P$ along $L$.

According to a well-known theorem of Cram\'er and Wold \cite{CW36}, 
if $P,Q$ are two Borel probability measures on $\mathbb{R}^d$, 
and if $P_L=Q_L$ for all lines $L$, then $P=Q$.
In other words, a probability measure on $\mathbb{R}^d$ 
is determined by its complete set of one-dimensional marginal
distributions.

There are several extensions of this theorem,
in which one assumes more about the nature of the measures $P,Q$ 
and less about the set of lines $L$ for which $P_L=Q_L$.
For example, if $P$ and $Q$ have  moment generating functions that are finite in a neighbourhood of the origin,
and if $P_L=Q_L$ for all lines $L$ in a set of positive measure on the unit sphere,
then $P=Q$. 
Articles on this subject include those of R\'enyi \cite{Re52},
Gilbert \cite{Gi55},  B\'elisle--Mass\'e--Ransford \cite{BMR97}
and Cuesta-Albertos--Fraiman--Ransford \cite{CFR07}.

If one assumes yet more about $P$ and $Q$, then it is even possible
to differentiate between them using only a finite set of projections.
Heppes \cite{He56} showed that, if $P$ and $Q$ are supported on a finite set of cardinality $k$,
and if $H_1,\dots,H_{k+1}$ are vector subspaces such that $H_i^\perp\cap H_j^\perp=\{0\}$ whenever $i\ne j$,
then $P_{H_j}=Q_{H_j}$ for all~$j$ implies that $P=Q$.

Another result of this kind, due to Gr\"ochenig and Jaming
\cite{GJ20}, is that a probability measure supported
on a quadratic hypersurface in $\mathbb{R}^d$ is determined
by its projections onto  two generic hyperplanes. The proof is based on the 
notion of a Heisenberg uniqueness pair, introduced in \cite{HM11}.

Our goal in this note is to establish an analogue
of these results for a certain family of continuous distributions,
namely the so-called elliptical distributions. Here is the definition.

\begin{definition}\label{D:elliptical}
A Borel measure $P$ on $\mathbb{R}^d$ is an elliptical distribution
on $\mathbb{R}^d$ if its characteristic function has the form
\begin{equation}\label{E:elliptical}
\phi_P(\xi)=e^{i\mu\cdot\xi}\psi(\xi^{\top} \Sigma \xi),
\quad \xi\in\mathbb{R}^d,
\end{equation}
where $\Sigma$ is a real positive semi-definite $d\times d$ matrix, 
$\mu$ is a vector in $\mathbb{R}^d$, and $\psi:[0,\infty)\to\mathbb{C}$ is a continuous function.
The measure $P$  is said to be centred if $\mu=0$.
\end{definition}

 If $P$ is an elliptical distribution with  finite second moments, then $\mu$ represents the mean and $\Sigma$ the covariance matrix. However, there are  elliptical distributions even whose first moments are infinite.

The most important  elliptical distributions are surely the Gaussian distributions. They correspond to taking $\psi(t)=e^{-t/2}$ in \eqref{E:elliptical}. However, there are numerous  other  examples of interest,
including multivariate Student distributions, Cauchy distributions, Bessel distributions, logistic distributions, stable laws, Kotz-type distributions, and multi-uniform distributions.
For background on the general theory of elliptical distributions, we refer to \cite{CHS81} and \cite{FKN90}. The latter reference also contains a more complete list of examples (see \cite[Table~3.1]{FKN90}).

To state our results, we need a further definition.

\begin{definition}
A set $S$ of vectors in $\mathbb{R}^d$ is a 
symmetric-matrix uniqueness set (or sm-uniqueness set for short)
if the only real symmetric $d\times d$ matrix $A$ satisfying
$x^{\top} Ax=0$ for all $x\in S$ is the zero matrix.
\end{definition}

For example, if $v_1,\dots,v_d$ are linearly independent vectors in $\mathbb{R}^d$
and if we define $S:=\{v_i+v_j:1\le i\le j\le d\}$, then $S$ is an sm-uniqueness set.
Also, every sm-uniqueness set must contain at least $(d^2+d)/2$ elements,
so the example $S$ above is minimal. We shall justify these statements in
\S\ref{S:uniqueness} below, where sm-uniqueness sets will be discussed in more detail.

We can now state our first main theorem.
Given $x\in\mathbb{R}^d\setminus\{0\}$, we denote by $\langle x\rangle$
the one-dimensional subspace spanned by $x$.

\begin{theorem}\label{T:CWelliptical}
Let $S\subset\mathbb{R}^d$ be a symmetric-matrix uniqueness set.
If $P,Q$ are elliptical measures on $\mathbb{R}^d$ such that $P_{\langle x\rangle}=Q_{\langle x\rangle}$ for all $x\in S$, then $P=Q$.
\end{theorem}

The following corollary is worthy of note. It follows from the fact that, in $\mathbb{R}^2$, any set of three vectors, none of which is a multiple of the others, forms an sm-uniqueness set.

\begin{corollary}\label{C:CWelliptical}
An elliptical distribution in $\mathbb{R}^2$ is determined by its marginals along any three distinct lines.
\end{corollary}

\begin{remarks}
(i) Perhaps the main interest of the theorem lies in the fact that,
to differentiate between two elliptical distributions,
only  finitely many projections are needed.
In that sense, the  theorem is  an analogue of the theorems 
of Heppes and of Gr\"ochenig--Jaming, mentioned earlier. However, there are also important differences between those results and ours. The  results  of Heppes and of Gr\"ochenig--Jaming treat measures supported on certain types of subsets of $\mathbb{R}^d$, whereas Theorem~\ref{T:CWelliptical} treats continuous ones. Also, in their theorems,  the projections are onto general subspaces $H$, whereas in our result the projections are onto lines $L$, i.e.,  only one-dimensional projections are needed.

(ii) Among the numerous variants of the Cram\'er--Wold theorem
for continuous distributions,
 Theorem~\ref{T:CWelliptical} is unusual (maybe even unique) in that no assumption is made about the finiteness of moments of $P$ and $Q$. As mentioned earlier,
there are elliptical distributions whose first moments are infinite, for example the multivariate Cauchy distributions.

(iii) In Theorem~\ref{T:CWelliptical}, it is important to assume that both $P$ and $Q$ are elliptical distributions. If we suppose merely that one of them is elliptical, then the result no longer holds.
(In the terminology of \cite{BMR97},  Theorem~\ref{T:CWelliptical} is not a strong-determination result.)
Indeed, given any finite set $\mathcal{H}$
of $(d-1)$-dimensional subspaces of $\mathbb{R}^d$, there exist probability measures $P,Q$ on $\mathbb{R}^d$ with $P$ Gaussian and $P_H=Q_H$ for all $H\in\mathcal{H}$, but $P\ne Q$. Such examples were constructed by
Hamedani and Tata \cite{HT75} in the case $d=2$, and by
Manjunath and Parthasarathy \cite{MP12} for general~$d$.
\end{remarks}

We now turn to the question of sharpness. 
The following result shows that Theorem~\ref{T:CWelliptical} is optimal in a certain sense.
In particular, it explains why symmetric-matrix uniqueness sets enter the picture.

We say that a Borel probability measure $P$ on $\mathbb{R}^d$ is non-degenerate
 if it is not supported in any hyperplane in $\mathbb{R}^d$. 

\begin{theorem}\label{T:CWellipticalsharp}
Let $P$ be a non-degenerate elliptical distribution on $\mathbb{R}^d$.
Let $S\subset\mathbb{R}^d$, and suppose that $S$ is not a symmetric-matrix uniqueness set.
Then there exists a non-degenerate elliptical distribution $Q$ on $\mathbb{R}^d$
such that $P_{\langle x\rangle}=Q_{\langle x\rangle}$ for all $x\in S$, but $P\ne Q$.
\end{theorem}

As remarked earlier, a symmetric-matrix uniqueness set in $\mathbb{R}^d$  must contain at least $(d^2+d)/2$ elements.
We thus obtain the following corollary.

\begin{corollary}\label{C:CWellipticalsharp}
Let $P$ be a non-degenerate elliptical distribution on $\mathbb{R}^d$.
Let $\mathcal{L}$ be a set of lines in $\mathbb{R}^d$ containing strictly fewer than $(d^2+d)/2$ lines.
Then there exists a non-degenerate elliptical distribution $Q$ on $\mathbb{R}^d$
such that $P_{L}=Q_{L}$ for all $L\in\mathcal{L}$, but $P\ne Q$.
\end{corollary}

The rest of the paper is organized as follows.
In \S\ref{S:uniqueness} we discuss in more detail the notion of symmetric-matrix uniqueness sets.
The proofs of our two main results, Theorems~\ref{T:CWelliptical} and \ref{T:CWellipticalsharp}
are presented in \S\ref{S:proof1} and \S\ref{S:proof2} respectively.
In \S\ref{S:statistics} we use Theorem~\ref{T:CWelliptical} to derive a statistical test for equality of 
elliptical distributions, and we carry out a small simulation study of the test.
In \S\ref{S:binary}, we give a further application, to binary classification.
Finally, in \S\ref{S:conclusion} we make some concluding remarks and pose a question.


\section{Symmetric-matrix uniqueness sets}\label{S:uniqueness}

Recall that a set $S$ of vectors in $\mathbb{R}^d$ is a 
symmetric-matrix uniqueness set or sm-uniqueness set
if the only real symmetric $d\times d$ matrix $A$ satisfying
$x^{\top} Ax=0$ for all $x\in S$ is the zero matrix.
We now examine  these sets in more detail,
beginning with the following simple result.

\begin{proposition}\label{P:uniquenessexample}
Let $v_1,\dots,v_d$ be  linearly independent vectors in $\mathbb{R}^d$, and let
\[
S:=\{v_j+v_k: 1\le j\le k\le d\}.
\]
Then $S$ is an sm-uniqueness set. 
\end{proposition}

\begin{proof}[\textbf{\upshape Proof:}]
If $A$ is any symmetric matrix, then, for all $j,k$, we have
\[
2v_j^\top A v_k=(v_j+v_k)^\top A(v_j+v_k)-\frac{1}{4}(v_j+v_j)^\top A(v_j+v_j)-
\frac{1}{4}(v_k+v_k)^\top A(v_k+v_k).
\]
Hence, if $x^\top Ax=0$ for all $x\in S$, then $v_j^\top Av_k=0$ for all $j,k$,
and so $A=0$.
\end{proof}

\begin{corollary}
If $v_1,v_2,v_3\in\mathbb{R}^2$ and no $v_j$ is a multiple of any other, then $\{v_1,v_2,v_3\}$ is an
sm-uniqueness set.
\end{corollary}

\begin{proof}[\textbf{\upshape Proof:}]
Clearly $v_1,v_2$ are linearly independent. Also, replacing them by suitable non-zero multiples of themselves,
we can suppose that $v_3=v_1+v_2$. The result therefore follows from Proposition~\ref{P:uniquenessexample}.
\end{proof}

In higher dimensions, criteria for sm-uniqueness
sets are more complicated, though we shall derive some in
Proposition~\ref{P:criterion} and Corollary~\ref{C:criterion} below.
However, for statistical applications, it suffices merely to have some concrete  examples of sm-uniqueness sets. The next result furnishes a  particularly simple example.
\begin{corollary}\label{C:concrete}
Let $S$ be the set consisting of those vectors in $\mathbb{R}^d$ with either one or two coordinates equal to $1$ and all the other coordinates equal to $0$.
Then $S$ is an sm-uniqueness set.
\end{corollary}
\figurename~\ref{F:arrow} illustrates the set $S$ in dimensions $2$ and $3$.

\begin{figure}[H]
\begin{minipage}{0.4\linewidth}
\begin{center}
\begin{tikzpicture}[scale=0.7]
\draw[-][dashed] (4,0) -- (4,4);
\draw[-][dashed] (0,4) -- (4,4);
\draw[->][very thick] (0,0) -- (0,4);
\draw[->][very thick] (0,0) -- (4,0);
\draw[->][very thick] (0,0) -- (4,4);
\draw (4,0) node[below]{$(1,0)$};
\draw (0,4) node[left]{$(0,1)$};
\draw (4,4) node[right]{$(1,1)$};
\end{tikzpicture}
\end{center}
\end{minipage}
\begin{minipage}{0.5\linewidth}
\begin{center}
\begin{tikzpicture}[scale=0.6]
\draw[-][dashed] (4,0) -- (4,4);
\draw[-][dashed] (0,4) -- (4,4);
\draw[-][dashed] (0,4) -- (2,5);
\draw[-][dashed] (4,4) -- (6,5);
\draw[-][dashed] (2,5) -- (6,5);
\draw[-][dotted] (2,1) -- (2,5);
\draw[-][dotted] (2,1) -- (6,1);
\draw[-][dashed] (4,0) -- (6,1);
\draw[-][dashed] (6,1) -- (6,5);
\draw[->][very thick] (0,0) -- (0,4);
\draw[->][very thick] (0,0) -- (4,0);
\draw[->][very thick] (0,0) -- (4,4);
\draw[->][very thick] (0,0) -- (2,1);
\draw[->][very thick] (0,0) -- (2,5);
\draw[->][very thick] (0,0) -- (6,1);
\draw (4,0) node[below]{$(1,0,0)$};
\draw (0,4) node[left]{$(0,0,1)$};
\draw (4,4) node[right]{$(1,0,1)$};
\draw (2,1) node[right]{$(0,1,0)$};
\draw (6,1) node[right]{$(1,1,0)$};
\draw (2,5) node[left]{$(0,1,1)$};
\end{tikzpicture}
\end{center}
\end{minipage}
\caption{Visualization of the vectors of the set $S$ from Corollary~\ref{C:concrete}  in dimensions  two and three.}
\label{F:arrow}
\end{figure}
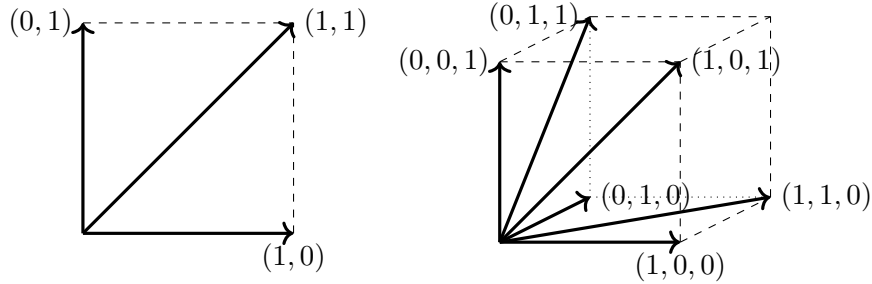

\begin{proof}[\textbf{\upshape Proof of Corollary~\ref{C:concrete}:}]
In Proposition~\ref{P:uniquenessexample},
take $v_j=e_j$, where $\{e_1,\dots,e_d\}$ is the standard unit vector basis of $\mathbb{R}^d$. By that result, the union of the two sets $\{e_j+e_k:1\le j<k\le d\}$ and $\{2e_j:1\le j\le d\}$ is an sm-uniqueness set. It clearly makes no difference if,
in the second set, we replace $2e_j$ by $e_j$. 
\end{proof}

The sets $S$ in Proposition~\ref{P:uniquenessexample}
and Corollary~\ref{C:concrete}
contain $(d^2+d)/2$ elements.
The next result shows that this number is minimal.

\begin{proposition} 
An sm-uniqueness set in $\mathbb{R}^d$ contains at least $(d^2+d)/2$ elements.	
\end{proposition}

\begin{proof}[\textbf{\upshape Proof:}]
Let $v_1,\dots,v_m\in\mathbb{R}^d$
with $m<(d^2+d)/2$.  The vector space $V$ of all $d\times d$ symmetric matrices has dimension $(d^2+d)/2$, so the linear map $A\mapsto (v_1^\top Av_1,\dots,v_m^\top Av_m):V\to\mathbb{R}^m$
must have a non-zero kernel. Thus there exists $A\ne0$
such that $v_j^\top Av_j=0$ for $j\in\{1,\dots,m\}$. Hence $\{v_1,\dots,v_m\}$ is not an sm-uniqueness set.
\end{proof}

\begin{proposition}\label{P:span}
If $S$ is an sm-uniqueness set in $\mathbb{R}^d$, then $S$  spans $\mathbb{R}^d$. 
\end{proposition}

\begin{proof}[\textbf{\upshape Proof:}]
Suppose, on the contrary, that $S$ does not span $\mathbb{R}^d$.
Then we can find a unit vector $v\in S^\perp$.
Let $R$ be a $d\times d$ rotation matrix such that $Rv=e_1$,
 and let $A$ be the $d\times d$ symmetric matrix given by 
 $A:=R^\top DR$, where $D:=\diag(1,0,\dots,0)$.
If $x\in S$, then $x\perp v$, so $Rx\perp Rv=e_1$, so 
$Rx\in\spn\{e_2,\dots,e_d\}$ and hence $(Rx)^\top D(Rx)=0$.
In other words, $x^\top Ax=0$ for all $x\in S$.
On the other hand, $v^\top Av=(Rv)^\top DRv=e_1^\top De_1=1$, so $A\ne0$.
We conclude that $S$ is not an sm-uniqueness set.
\end{proof}

The next result is not really needed in what follows,
but it provides a necessary and sufficient condition for a 
set of  vectors in $\mathbb{R}^d$ to form an sm-uniqueness set.

\begin{proposition}\label{P:criterion}
Given a vector $x\in\mathbb{R}^d$, say $x=(x_1,\dots,x_d)$,
let $\hat{x}$ be the upper-triangular $d\times d$  matrix with entries
$\hat{x}_{ij}:=x_ix_j~( 1\le i\le j\le d)$.
A set $S\subset\mathbb{R}^d$  is an sm-uniqueness
set  if and only if 
$\{\hat{x}:x\in S\}$
spans  the  space of upper-triangular $d\times d$ matrices.
\end{proposition}

\begin{proof}[\textbf{\upshape Sketch of Proof:}]
Denote by $\sym^2(\mathbb{R}^d)$ the symmetric tensor product of $\mathbb{R}^d$ with itself.
For each symmetric $d\times d$ matrix $A$, 
there is a unique linear functional $\tilde{A}:\sym^2(\mathbb{R}^d)\to\mathbb{R}$ such that $\tilde{A}(x\otimes x)=x^\top Ax$ for all $x\in\mathbb{R}^d$.
Thus a set $S\subset \mathbb{R}^d$ is an sm-uniqueness
set  iff $\{x\otimes x:x\in S\}$
spans $\sym^2(\mathbb{R}^d)$.

Let $e_1,\dots,e_d$ be the standard unit basis of $\mathbb{R}^d$.
Then $\sym^2(\mathbb{R}^d)$ has a basis $\{E_{ij}:1\le i\le j\le d\}$,
where $E_{ij}:=(e_i\otimes e_j+e_j\otimes e_i)/2$. 
Expressing $x\otimes x$ in terms of this basis, we have
\[
x\otimes x=\sum_{1\le i\le j\le d}x_ix_j E_{ij}.
\]
The result follows easily from this.
\end{proof}

Since the space of upper-triangular $d\times d$ matrices has dimension $(d^2+d)/2$,
we deduce the following corollary.

\begin{corollary}\label{C:criterion}
Let $S$ be a set of $(d^2+d)/2$ vectors in $\mathbb{R}^d$.
For each $x\in S$, define $\hat{x}$ as in
Proposition~\ref{P:criterion}.
Then $S$ is an sm-uniqueness set  if and only if
the matrices $\{\hat{x}:x\in S\}$ are linearly independent.
\end{corollary}

\begin{remark}
If we view an upper-triangular $d\times d$ matrix as a column vector of length $D:=(d^2+d)/2$,
then Corollary~\ref{C:criterion} becomes a criterion expressed in terms of the linear independence of $D$ vectors in $\mathbb{R}^D$, which can in turn be reformulated as the non-vanishing of a $D\times D$ determinant. This furnishes a
systematic method of determining whether or not a given set of $(d^2+d)/2$ vectors in $\mathbb{R}^d$ is an sm-uniqueness set.
\end{remark}


\section{Proof of Theorem~\ref{T:CWelliptical}}\label{S:proof1}

We break the proof into a series of lemmas. 
The first of these is fairly standard.

\begin{lemma}\label{L:cf}
Let $P,Q$ be Borel probability measures on $\mathbb{R}^d$
and let $L$ be a line in $\mathbb{R}^d$. Then $P_{L}=Q_{L}$
if and only if the characteristic functions of $P,Q$ satisfy
 $\phi_P(\xi)=\phi_Q(\xi)$ for all $\xi\in L$.
\end{lemma}

\begin{proof}[\textbf{\upshape Proof:}]
Let $\xi\in L$. Then $x\cdot\xi=\pi_L(x)\cdot\xi$ for all $x\in\mathbb{R}^d$.
Consequently, 
\[
\phi_{P}(\xi)=\int_{\mathbb{R}^d}e^{ix\cdot\xi}\,dP(x)
=\int_{\mathbb{R}^d}e^{i\pi_L(x)\cdot\xi}\,dP(x)
=\int_{L}e^{i y\cdot\xi}\,dP_L(y),
\]
and similarly for $Q$. Hence, if $P_L=Q_L$, 
then $\phi_{P}(\xi)=\phi_{Q}(\xi)$. 

Conversely, if $\phi_P(\xi)=\phi_Q(\xi)$ for all $\xi\in L$,
then the calculation above shows that $P_L$ and $Q_L$ have the 
same characteristic function, and consequently $P_L=Q_L$.
\end{proof}

\begin{lemma}\label{L:mu1=mu2}
Let $P,Q$ be elliptical distributions, with characteristic functions
\[
\phi_P(\xi)=e^{i\mu_1\cdot\xi}\psi_1(\xi^\top \Sigma_1\xi)
\quad\text{and}\quad
\phi_Q(\xi)=e^{i\mu_2\cdot\xi}\psi_2(\xi^\top \Sigma_2\xi),
\quad \xi\in\mathbb{R}^d.
\]
If  $P_{\langle x\rangle}=Q_{\langle x\rangle}$ for all $x$ in some spanning set of $\mathbb{R}^d$,
then $\mu_1=\mu_2$.
\end{lemma}

\begin{proof}[\textbf{\upshape Proof:}]
By Lemma~\ref{L:cf}, if $P_{\langle x\rangle}=Q_{\langle x\rangle}$,
then $\phi_P=\phi_Q$ on $\langle x\rangle$,
in other words $\phi_P(tx)=\phi_Q(tx)$ for all $t\in\mathbb{R}$.
Recalling the form of $\phi_P,\phi_Q$, we obtain
\begin{equation}\label{E:basiceqn}
e^{it\mu_1\cdot x}\psi_1(t^2x^\top \Sigma_1 x)=
e^{it\mu_2\cdot x}\psi_2(t^2 x^\top \Sigma_2x),
\quad t\in\mathbb{R}.
\end{equation}
Since $\phi_P,\phi_Q$ are characteristic functions,
they are both  equal to $1$ at the origin.
It follows that $\psi_1(0)=\psi_2(0)=1$. 
As $\psi_1,\psi_2$ are continuous, there exists $\delta>0$
such that both $\psi_1(t^2x^\top \Sigma_1x)\ne0$ and  $\psi_2(t^2x^\top\Sigma_2x)\ne0$ for all $t\in[-\delta,\delta]$.
For each such $t$, we may divide equation \eqref{E:basiceqn}
by the same equation in which $t$ is replaced by $-t$. This gives
\[
e^{2it\mu_1\cdot x}=e^{2it\mu_2\cdot x}, \quad t\in[-\delta,\delta].
\]
Differentiating with respect to $t$ and setting $t=0$,
we obtain $\mu_1\cdot x=\mu_2\cdot x$.
Finally, if this holds for all $x$ in a spanning set of $\mathbb{R}^d$,
then $\mu_1=\mu_2$.
\end{proof}

\begin{lemma}\label{L:samepsi}
Let $P,Q$ be centred elliptical distributions on $\mathbb{R}^d$
such that $P_{\langle x\rangle}=Q_{\langle x\rangle}$
for all $x$ in some spanning set of $\mathbb{R}^d$. 
Then there exists a continuous function $\psi:[0,\infty)\to\mathbb{C}$
and positive semi-definite $d\times d$ matrices 
$\Sigma_1,\Sigma_2$ such that
\begin{equation}\label{E:samepsi}
\phi_P(\xi)=\psi(\xi^\top\Sigma_1\xi),
\quad
\phi_Q(\xi)=\psi(\xi^\top\Sigma_2\xi),
\quad \xi\in\mathbb{R}^d.
\end{equation}
\end{lemma}

\begin{proof}[\textbf{\upshape Proof:}]
By assumption, there exist continuous functions 
$\psi_1,\psi_2:[0,\infty)\to\mathbb{C}$ and positive semi-definite
$d\times d$ matrices $\Sigma_1',\Sigma_2'$ such that
\[
\phi_P(\xi)=\psi_1(\xi^\top\Sigma_1'\xi),
\quad
\phi_Q(\xi)=\psi_2(\xi^\top\Sigma_2'\xi),
\quad \xi\in\mathbb{R}^d.
\]
The point of the lemma is to show that, adjusting $\Sigma_1',\Sigma_2'$ if necessary, we may take $\psi_1=\psi_2$.

If $\Sigma_1'=\Sigma_2'=0$, then  we may as well take $\psi_1=\psi_2\equiv1$. 

Otherwise, at least one of $\Sigma_1',
\Sigma_2'$ is non-zero, say $\Sigma_1'\ne0$.
Since $\Sigma_1'$ is positive semi-definite, any spanning
set of $\mathbb{R}^d$ contains a vector $x$ such that $x^\top\Sigma_1'x>0$.
Hence, there exists $x_0\in\mathbb{R}^d$ such that both $P_{\langle x_0\rangle}=Q_{\langle x_0\rangle}$ and $x_0^\top\Sigma_1'x_0>0$.
By Lemma~\ref{L:cf},
\[
\psi_1(t^2x_0^\top\Sigma_1' x_0)= \psi_2(t^2 x_0^\top\Sigma_2'x_0),
\quad t\in\mathbb{R}.
\]
It follows that
\[
\psi_1(s)=\psi_2(as), \quad s\ge0,
\]
where $a:=(x_0^\top\Sigma_2'x_0)/(x_0^\top\Sigma_1'x_0)\ge0$. 
Feeding this information into the formula for $\phi_P$, we obtain
\[
\phi_P(\xi)=\psi_2(\xi^\top(a\Sigma_1')\xi),
\quad \xi\in\mathbb{R}^d.
\]
Thus \eqref{E:samepsi} holds with $\psi:=\psi_2$ and $\Sigma_1:=a\Sigma_1'$ and $\Sigma_2:=\Sigma_2'$.
\end{proof}

\begin{lemma}\label{L:end}
Let $P,Q$ be probability measures on $\mathbb{R}^d$ with characteristic functions given by \eqref{E:samepsi}.
Let $S$ be an sm-uniqueness set,
and suppose that $P_{\langle x\rangle}=Q_{\langle x\rangle}$ for all $x\in S$. Then $P=Q$.
\end{lemma}

\begin{proof}[\textbf{\upshape Proof:}]
We show that either $\Sigma_1=\Sigma_2$ or $\psi$ is constant.
Either way, this gives that $\phi_P=\phi_Q$ and hence that $P=Q$.

Suppose then that $\Sigma_1\ne\Sigma_2$.
As $S$ is an sm-uniqueness set,
there exists $x_0\in S$ such that 
$x_0^\top(\Sigma_1-\Sigma_2)x_0\ne0$.
Exchanging the roles of $P,Q$ if necessary, we can suppose that
\[
x_0^\top\Sigma_1x_0>x_0^\top\Sigma_2 x_0\ge0.
\]
Since $P_{\langle x_0\rangle}=Q_{\langle x_0\rangle}$, Lemma~\ref{L:cf} gives
\[
\psi(t^2x_0^\top\Sigma_1 x_0)=\psi(t^2x_0^\top\Sigma_2 x_0),
\quad t\in\mathbb{R}.
\]
It follows that
\[
\psi(s)=\psi(as), \quad s\ge0,
\]
where $a:=(x_0^\top\Sigma_2x_0)/(x_0^\top\Sigma_1 x_0)\in[0,1)$.
Iterating this equation, and using the continuity of $\psi$, we obtain
\[
\psi(s)=\psi(as)=\psi(a^2s)=\cdots=\psi(a^ns)\underset{n\to\infty}\longrightarrow \psi(0), \quad s\ge0.
\]
Therefore $\psi$ is constant, as claimed.
\end{proof}

\begin{proof}[\textbf{\upshape Completion of proof of Theorem~\ref{T:CWelliptical}:}]
Let $P,Q$ be elliptical distributions on $\mathbb{R}^d$, with characteristic
functions 
\[
\phi_P(\xi)=e^{i\mu_1\cdot\xi}\psi_1(\xi^\top\Sigma_1\xi)
\quad\text{and}\quad
\phi_Q(\xi)=e^{i\mu_2\cdot\xi}\psi_2(\xi^\top\Sigma_2\xi).
\]
Suppose that
$P_{\langle x\rangle}=Q_{\langle x\rangle}$ for all $x\in S$,
where $S\subset\mathbb{R}^d$ is an sm-uniqueness set.
Our goal is to prove that $P=Q$.

By Proposition~\ref{P:span}, $S$ is  a spanning set for $\mathbb{R}^d$.
By Lemma~\ref{L:mu1=mu2},  it follows that $\mu_1=\mu_2$. Translating both $P,Q$ by $-\mu_1$,
we may suppose that both measures are centred.
Lemma~\ref{L:samepsi}  then shows that $P,Q$ have characteristic functions given by \eqref{E:samepsi},
and finally Lemma~\ref{L:end} gives  $P=Q$.
\end{proof}


\section{Proof of Theorem~\ref{T:CWellipticalsharp}}\label{S:proof2}

Once again, we break the proof up into lemmas.
The first lemma characterizes those elliptical distributions on $\mathbb{R}^d$ 
that are non-degenerate 
(i.e., not supported on any hyperplane in $\mathbb{R}^d$).

\begin{lemma}\label{L:nondegen}
Let $P$ be an elliptical distribution on $\mathbb{R}^d$ with characteristic
function
\[
\phi_P(\xi)=e^{i\mu\cdot\xi}\psi(\xi^\top\Sigma\xi),
\quad \xi\in\mathbb{R}^d.
\]
Then $P$ is non-degenerate if and only if $\psi$ is non-constant and $\Sigma$ is strictly positive definite.
\end{lemma}

\begin{proof}[\textbf{\upshape Proof:}]
Suppose that $P$ is supported on the hyperplane
$H:=\{x\in\mathbb{R}^d: y\cdot x=c\}$, 
where $y\in\mathbb{R}^d\setminus\{0\}$ and $c\in\mathbb{R}$.
Then   $P_{\langle y\rangle} =\delta_{x_0}$, where $x_0$ is the unique point in $\langle y\rangle \cap H$. By Lemma~\ref{L:cf}, 
\[
\phi_P(ty)=e^{ity\cdot x_0}, \quad t\in\mathbb{R}.
\]
Recalling the form of $\phi_P$, we deduce that
\[
e^{it\mu\cdot y}\psi(t^2 y^\top\Sigma y)=e^{ity\cdot x_0}, \quad t\in\mathbb{R}.
\]
The argument used to prove Lemma~\ref{L:mu1=mu2} shows that
$\mu\cdot y=x_0\cdot y$, and hence that
\[
\psi(t^2 y^\top\Sigma y)=1, \quad t\in\mathbb{R}.
\]
This in turn implies that either $\psi\equiv 1$ or $y^\top\Sigma y=0$.
Thus either $\psi$ is constant or $\Sigma$ is not strictly positive definite.

The converse is proved by running the same argument backwards.
\end{proof}

\begin{lemma}\label{L:Qexists}
Let $P$ be a non-degenerate elliptical distribution on $\mathbb{R}^d$, 
with characteristic function
\[
\phi_P(\xi)=e^{i\mu_1.\xi}\psi(\xi^\top\Sigma_1 \xi),
\quad \xi\in\mathbb{R}^d.
\]
Then, for each vector $\mu_2\in\mathbb{R}^d$ and each positive semi-definite $d\times d$ matrix $\Sigma_2$, there exists an elliptical distribution $Q$ on $\mathbb{R}^d$ with characteristic function
\[
\phi_Q(\xi)=e^{i\mu_2.\xi}\psi(\xi^\top\Sigma_2 \xi),
\quad \xi\in\mathbb{R}^d.
\]
\end{lemma}

\begin{proof}[\textbf{\upshape Proof:}]
Since $\Sigma_1,\Sigma_2$ are positive semi-definite, we can write them as $\Sigma_1=S_1^\top S_1$ and $\Sigma_2=S_2^\top S_2$,
where $S_1,S_2$ are $d\times d$ matrices. Further, as $P$ is non-degenerate,  Lemma~\ref{L:nondegen} implies that $\Sigma_1$ is strictly positive definite, and so $S_1$ is invertible. Let $A:\mathbb{R}^d\to\mathbb{R}^d$ be the affine map defined by
\[
Ax:=S_1^{-1}S_2(x-\mu_1)+\mu_2, \quad x\in\mathbb{R}^d,
\]
and set $Q:=PA^{-1}$. We shall show that $\phi_Q$ has the required form.

First of all, we remark that, if $\tilde{P}$ and $\tilde{Q}$ denote the translates of $P,Q$ by $-\mu_1$ and $-\mu_2$ respectively, then
\[
\phi_{\tilde{P}}(\xi)=e^{-i\mu_1\cdot\xi}\phi_P(\xi)
\quad\text{and}\quad
\phi_{\tilde{Q}}(\xi)=e^{-i\mu_2\cdot\xi}\phi_Q(\xi),
\quad \xi\in\mathbb{R}^d.
\]
Further, $\tilde{Q}=\tilde{P}T^{-1}$, where $T:\mathbb{R}^d\to\mathbb{R}^d$ is the linear map given by 
\[
Tx:=S_1^{-1}S_2x, \quad x\in\mathbb{R}^d.
\]
By a calculation similar to that in the proof of Lemma~\ref{L:cf}, we have 
\[
\phi_{\tilde{P}T^{-1}}(\xi)=\phi_{\tilde{P}}(T\xi), \quad \xi\in\mathbb{R}^d.
\]
Putting all of this together, we get
\begin{align*}
e^{-i\mu_2\cdot\xi}\phi_Q(\xi)
&=\phi_{\tilde{Q}}(\xi)
=\phi_{\tilde{P}T^{-1}}(\xi)
=\phi_{\tilde{P}}(T\xi)\\
&=e^{-i\mu_1\cdot\xi}\phi_P(T\xi)
=e^{-i\mu_1\cdot\xi}\phi_P(S_1^{-1}S_2\xi).
\end{align*}
Recalling the form of $\phi_P$ and the fact that $\Sigma_j=S_j^\top S_j$ for $j\in\{1,2\}$, we obtain
\begin{align*}
e^{-i\mu_2\cdot\xi}\phi_Q(\xi)
&=e^{-i\mu_1\cdot\xi}e^{i\mu_1\cdot\xi}\psi(\xi^\top S_2^\top S_1^{-\top}\Sigma_1S_1^{-1}S_2\xi)\\
&=\psi(\xi^\top S_2^\top S_2\xi)=\psi(\xi^\top\Sigma_2\xi).
\end{align*}
Thus $\phi_Q$ does indeed have the required form.
\end{proof}

\begin{proof}[\textbf{\upshape Completion of the proof of Theorem~\ref{T:CWellipticalsharp}:}]
By Lemma~\ref{L:nondegen}, since $P$ is a non-degenerate elliptical distribution, 
its characteristic function is given by
\[
\phi_P(\xi)=e^{i\mu.\xi}\psi(\xi^\top\Sigma_1\xi), \quad \xi\in\mathbb{R}^d,
\]
where $\psi$ is non-constant and $\Sigma_1$ is strictly positive definite. 
As $S$ is not an sm-uniqueness set, there exists a
symmetric $d\times d$ matrix $A$ such that $x^\top Ax=0$ for all $x\in S$ but $A\ne0$. If $\epsilon>0$ is small enough, then $\Sigma_1+\epsilon A$ is also strictly positive definite. Fix such an $\epsilon$, and set $\Sigma_2:=\Sigma_1+\epsilon A$. By Lemma~\ref{L:Qexists},
there exists an elliptical distribution $Q$ on $\mathbb{R}^d$ such that
\[
\phi_Q(\xi)=e^{i\mu.\xi}\psi(\xi^\top\Sigma_2\xi), \quad \xi\in\mathbb{R}^d.
\]
By Lemma~\ref{L:nondegen} again,  $Q$ is non-degenerate, since $\psi$ is non-constant  $\Sigma_2$ is positive definite. Also, if $x\in S$, then 
\[
x^\top\Sigma_2x=x^\top\Sigma_1x+x^\top Ax=x^\top\Sigma_1x,
\]
so $\phi_Q(tx)=\phi_P(tx)$ for all $t$.
Hence $P_{\langle x\rangle}=Q_{\langle x\rangle}$ for all $x\in S$.
On the other hand, since $\Sigma_1\ne\Sigma_2$ and $\psi$ is non-constant, the argument of the proof of Lemma~\ref{L:end} shows that $P\ne Q$.
\end{proof}


\section{Application to testing for equality of multivariate distributions}\label{S:statistics}
\subsection{A Kolmogorov--Smirnov test for elliptical distributions}

Several of the variants of the Cram\'er--Wold theorem mentioned in the introduction are useful in deriving
statistical tests for the equality of multivariate distributions (see, e.g., \cite{CFR06, CF09, FMR22a, FMR23}).
Theorem~\ref{T:CWelliptical} is no exception. 
In this section, we propose a test for the one- and two-sample problems for elliptical distributions, 
based on Theorem~\ref{T:CWelliptical}. 
The problem of goodness-of-fit testing for elliptical distributions is discussed in several recent articles (see, e.g.,  \cite{CJMZ22, DL20, HMS21}),  but always with more restrictive assumptions  than ours (supposing, for example, that the generator function  $\psi$ is known). We also remark that the idea of using projections of elliptical distributions is mentioned in \cite{No13}, though that article addresses a different problem.

Since the one- and two-sample problems are very similar, we describe  just  the two-sample problem.
Given two samples $\mathcal{X}_n:=\{X_1, \ldots, X_n\}\subset \mathbb R^d$
and $\mathcal{Y}_m:=\{Y_1, \ldots, Y_m\}\subset \mathbb R^d$ of multivariate elliptical distributions $P$ and $Q$,
 we consider the testing problem 
\[
\mathbf {H0}: P = Q \quad\text{vs}\quad P \neq Q,
\]
based on the samples $\mathcal{X}_n, \mathcal{Y}_m$.

Let $\{v_1, \dots v_D\}\subset\mathbb{R}^d$ be a fixed symmetric-matrix uniqueness set, where $D=(d^2 +d)/2$. 
Let $F_{n,\langle v_j\rangle}$ be the empirical distribution of $P_{\langle v_j\rangle}$  and 
$G_{m,\langle v_j\rangle}$ be the empirical distribution of $Q_{\langle v_j\rangle}$, for $j\in\{1, \ldots, D\}$.
We define a random-projection test (which we abbreviate as \textbf{RPT}) through the statistic,
\[
\sqrt{\frac{nm}{n+m}} KS_{n,m,D}
:= \sqrt{\frac{nm}{n+m}} \max_{j\in\{1, \ldots, D\}} \| F_{n,\langle v_j\rangle}-G_{m,\langle v_j\rangle}\|_{\infty}.
\]

Since  the statistic  is not distribution-free, 
in order to obtain the critical value for a level-$\alpha$ test, 
we approximate the distribution using bootstrap on the original samples $\mathcal{X}_n, \mathcal{Y}_m$
by generating a large enough number $B$ of values of $KS_{n,m}$, 
for each bootstrap sample
choosing $n$ vectors  from $\mathcal{X}_n$ and $m$ vectors from $\mathcal{Y}_m$, with replacement. 
See, for instance, \cite{HM88} for the two-sample bootstrap.
We then take as critical value $c^*_{\alpha}$, 
the $(1-\alpha)$-quantile of the empirical bootstrap sample, i.e., we reject the null hypothesis when 
\[
\sqrt{\frac{nm}{n+m}} KS_{n,m,D} > c^*_{\alpha}.
\]
The validity of the bootstrap in this case  follows from  \cite[Theorems~3 and~4]{Pr95}. 
Therefore the proposed test has asymptotic level $\alpha$.
Also it is consistent since, 
under the alternative hypothesis, $P\ne Q$, so, by Theorem~\ref{T:CWelliptical},
there exists a $j$ such that $P_{\langle v_j\rangle}\ne Q_{\langle v_j\rangle}$, so
$\max_{j\in\{1, \ldots, D\}} \|F_{n,\langle v_j\rangle}-G_{m,\langle v_j\rangle }\|_{\infty} >0$,
and thus
$\sqrt{\frac{nm}{n+m} } KS_{n,m,D}  \to \infty$.

We summarize our conclusions in the following proposition.

\begin{proposition}
Consider the test proposed above.
\begin{enumerate}[\normalfont(i)]
\item The bootstrap version of the test has asymptotic level $\alpha$.
\item Under the alternative,
\[
\sqrt{\frac{nm}{n+m} } KS_{n,m,D}  \to \infty, 
\]
i.e., the test is consistent.
\end{enumerate}
\end{proposition}


\subsection{A small simulation study}\label{S:sss}

In this subsection, we study the power of the proposed test for two elliptical distributions 
$F_1$ and $F_2$  belonging to the multivariate Student distribution in $\mathbb{R}^d$.  
More precisely, we suppose that
\[
F_i \sim t_{\nu_i} (\mu_i, \Sigma_i), \quad i\in\{1,2\}, 
\]
where $\nu_i \in \mathbb{N}$ are the numbers of degrees of freedom, 
$\mu_i \in \mathbb{R}^d$  are the means,
and  $\Sigma_i \in \mathbb R^{d \times d}$ are  the covariance matrices.
We consider three different scenarios. 

In the first scenario, we fix $\mu_i=0$ and $\Sigma_i = \mathbf{I}_{d\times d}$ (the $d\times d$ identity matrix), for $i\in\{1,2\}$ and we vary only the degrees of freedom:
 $\nu_1=1$ and $\nu_2 \in \{1,2,3,4\}$. The power functions are plotted in Fig.~\ref{Fig:esc1} for different sample sizes. 

In the second scenario, we vary only the mean value of one of the distributions, 
taking $\nu_i=2$ and $\Sigma_i= \mathbf{I}_{d\times d}$, for $i\in\{1,2\}$, 
while  $\mu_1=0$  and $\mu_2 \in [0,1/2]$. See Fig.~\ref{Fig:esc2}.
  
Lastly, in the third scenario, we vary only the covariance matrices. We fix $\nu_i=2$ and $\mu_i=0$ for $i\in\{1,2\}$,
 while $\Sigma_1= \mathbf{I}_{d \times d}$ and $\Sigma_2=  \mathbf{I}_{d \times d} + \theta \mathbf{1}_{d \times d}$, where $\mathbf{1}_{d \times d}$ denotes the matrix with value $1$ in all its entries,  and   $\theta \in [0,1]$.
See Fig.~\ref{Fig:esc3}. 
   
In all cases we generate iid samples  of equal sizes $n= n_1=n_2$, where $n \in \{100,500,1000\}$  and $d \in \{5,10,50\}$ for each scenario. The null distribution is approximated by bootstrap. 
For the power function, we report the proportion of rejections in 10000 replicates. 
The results are quite encouraging in all the three scenarios considered.
   Moreover, in the simulations where the sample size is large, the test size remained around the prespecified significance level ($\alpha=0.05$).

 \begin{figure}[H]
\centering
\subfloat{\includegraphics[width=\textwidth]{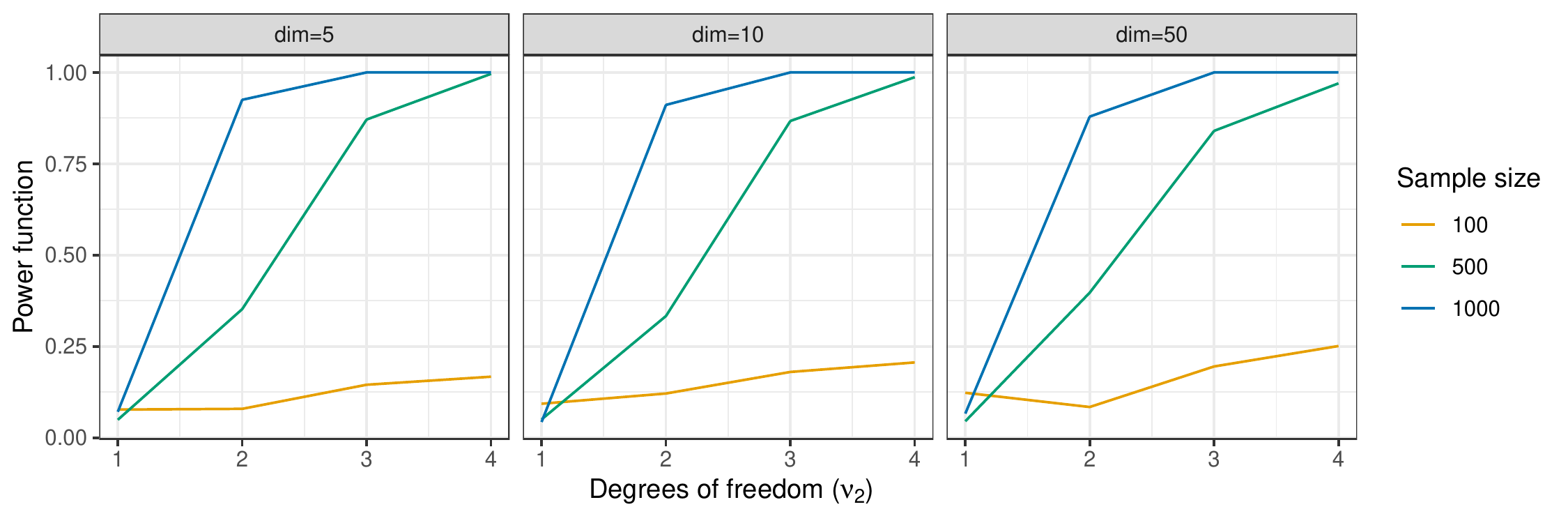}}
\caption{Power function of  test proposed in \S\ref{S:sss} for two elliptical distributions.
Scenario 1: changing the degrees of freedom and the sample sizes.} 
\label{Fig:esc1}
\end{figure}

\begin{figure}[H]
\centering
\subfloat{\includegraphics[width=\textwidth]{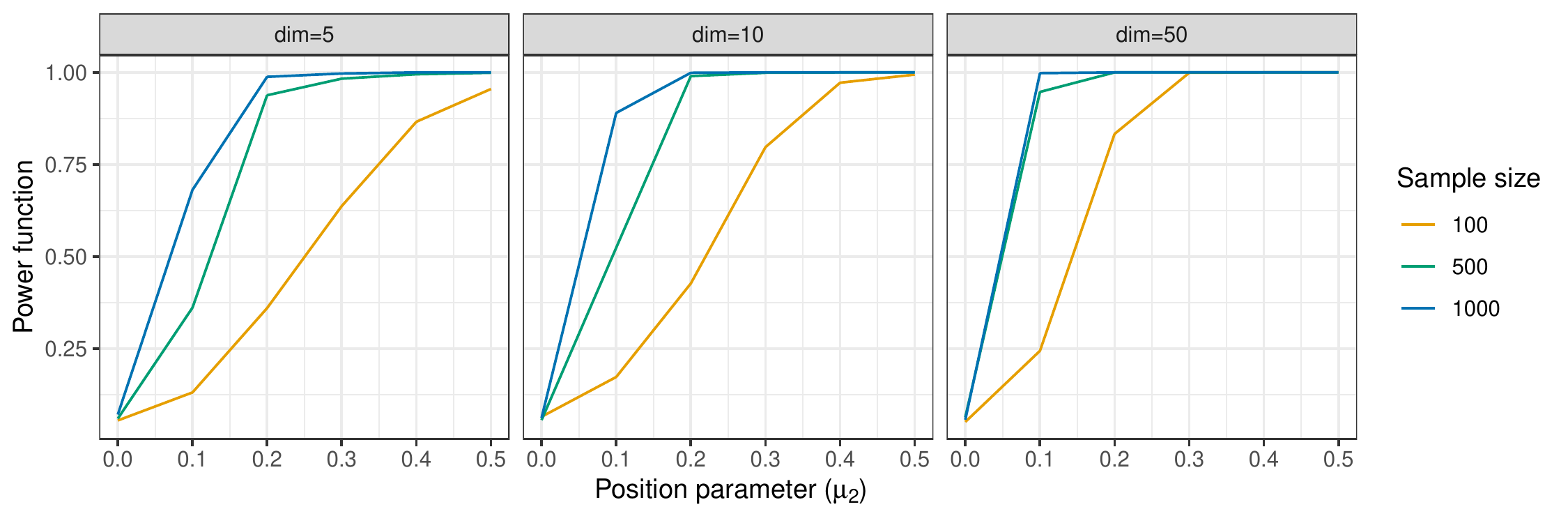}}
\caption{Power function of  test proposed in \S\ref{S:sss} for two elliptical distributions.
Scenario 2: changing  the mean value and the sample sizes.} 
\label{Fig:esc2}
\end{figure}

\begin{figure}[H]
\centering
\subfloat{\includegraphics[width=\textwidth]{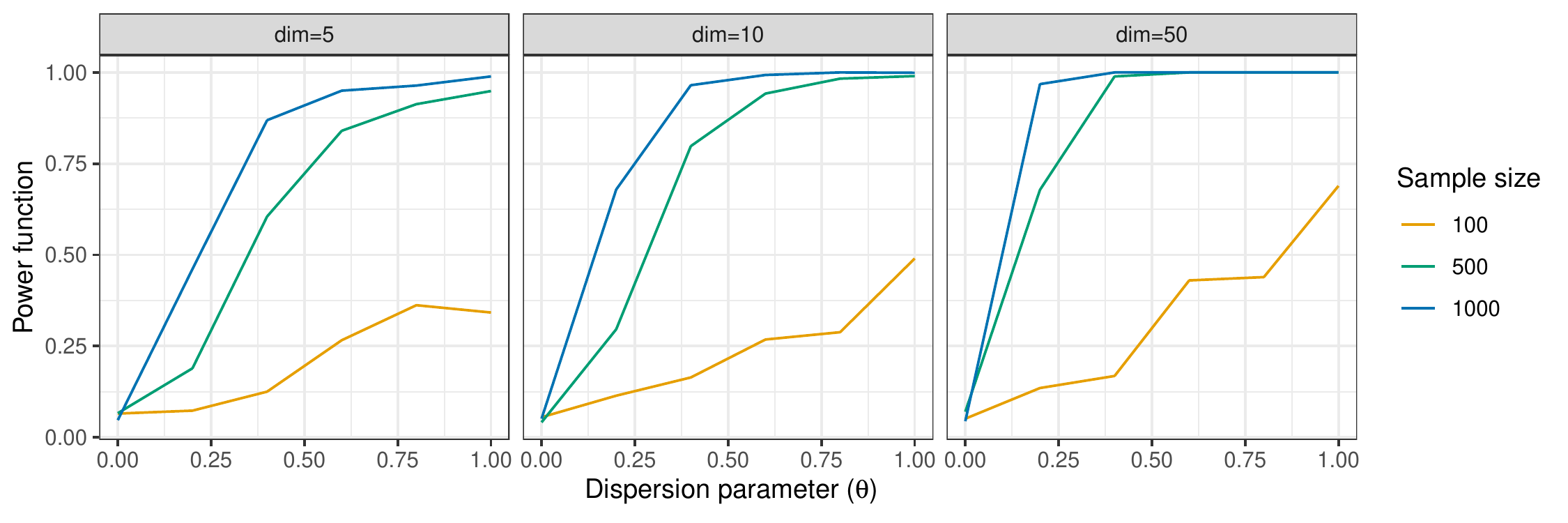}}
\caption{Power function of  test proposed in \S\ref{S:sss} for two elliptical distributions.
Scenario 3: changing  the covariance matrix as a function of $\theta$ and the sample sizes} 
\label{Fig:esc3}
\end{figure}


\subsection{Comparison with other tests}\label{S:comparison}
	
The following example compares the performance of the proposed test 
RPT for the equality of two distributions,
with simulated data from a mixture of distributions, with some other alternatives known from the literature.
The cases when the mixture has an elliptical distribution (and in particular a multivariate normal distribution) are considered, as well as the case  when the mixture distribution is not elliptical.  The performance of the RPT test is compared with two other proposals that impose different assumptions about data distributions. 
There are many test proposals for equality of multivariate normal distributions, 
but most of them are for equality of the mean---with known or unknown covariance matrix---or for equality of the covariance matrices---with known or unknown mean vector, but just a few more general. 
This is the reason why we have chosen these three particular competitors.

Let $\mu \in \mathbb{R}$ and  $\textbf{1}_d= (1,1, \dots,1)^\top \in \mathbb{R}^d$. Consider three
independent random vectors in $\mathbb{R}^d$, with distributions  given by
\begin{align*}
X_{1, \mu} &\sim \mathcal{N} \left( \mu \textbf{1}_d, \mathbf{I}_{d \times d} \right),\quad
X_{2, \mu}  \sim \mathcal{C} \left(\mu \textbf{1}_d, \mathbf{I}_{d \times d} \right),\\
X_{3, \mu} &\sim B \mathcal{N} \left( (1+\mu) \textbf{1}_d, \mathbf{I}_{d \times d} \right)+(1-B) \mathcal{N} \left( -(1+\mu) \textbf{1}_, \mathbf{I}_{d \times d} \right),
\end{align*}
with $B \sim \mathcal{B}(1/2)$, where $\mathcal{B}$ is the Bernoulli distribution, 
and where $\mathcal{N}$ and $\mathcal{C}$ are the Normal and Cauchy distributions in dimension $d$  respectively. 
Let $F_{1, \mu},F_{2, \mu}$ and $F_{3, \mu}$  be their respective probability distributions. 
We define the convex combination 
\[
F_{\mu}= \alpha_1  F_{1, \mu} + \alpha_2  F_{2, \mu} +\alpha_3  F_{3, \mu},
\]  
with  $\alpha_1,\alpha_2,\alpha_3 \geq 0$  and  $\alpha_1+\alpha_2+\alpha_3=1$.
Note that, if $\alpha_2=\alpha_3=0$, then the distribution is Gaussian; 
if  $\alpha_3=0$, then the distribution is elliptical; 
and  if  $\alpha_3 \neq 0$, then the distribution is not elliptical. 

In all cases we generate iid samples of equal sizes  $n_1 = n_2=100$,
$d=5$ and $F_1 \sim F_{\mu_1}$ and  $F_2 \sim F_{\mu_2}$ with $\mu_1=0$ and $\mu_2 \in \{0, 0.2, 0.4, 0.6 \}$. When $\mu_2=0$ we are under the null hypothesis
(that the distributions of $F_1$ and $F_2$ are the same). 
The power is studied in different scenarios obtained for different values of $\alpha=(\alpha_1,\alpha_2,\alpha_3$).

 The power of RPT is compared to:  
 \begin{description}
 \item[\rm LRTN] a simultaneous test for the mean and the covariance matrix for two samples under the assumption of normality (see \cite{LLL10} for details);  
 \item[\rm L2Norm] a more general test involving non-normal distributions   
 (see \cite{HN21} for details), and
 \item[\rm eDistance]  a non-parametric test developed in \cite{SR04}. 
 \end{description}
 
The L2Norm and eDistance tests are implemented in $\textsf{R}$ through the 
\texttt{sim2.2018HN} and \texttt{eqdist.etest} functions of the \texttt{SHT} and \texttt{energy} packages respectively. Note that, as mentioned in \cite{HN21}, the L2Norm test does not include in its assumptions all elliptical distributions. In the RPT test, the number of the $k$-nearest neighbors is chosen by cross-validation.

As shown in Fig.~\ref{F:Ex2}, if $\alpha_2 \neq 0$ the tests  LRTN and L2Norm perform very poorly.  This is to be expected, as they are not designed for heavy-tailed distributions. Also, if  $\alpha_2 \neq 0$ and $\alpha_3 =0$,
then the best-performing test in all cases is RPT. Even if the distributions are not elliptical, i.e., $\alpha_3 \neq 0$, RPT is a competitive test with eDistance.

\begin{figure}[htb]
\centering
\subfloat{\includegraphics[width=\textwidth]{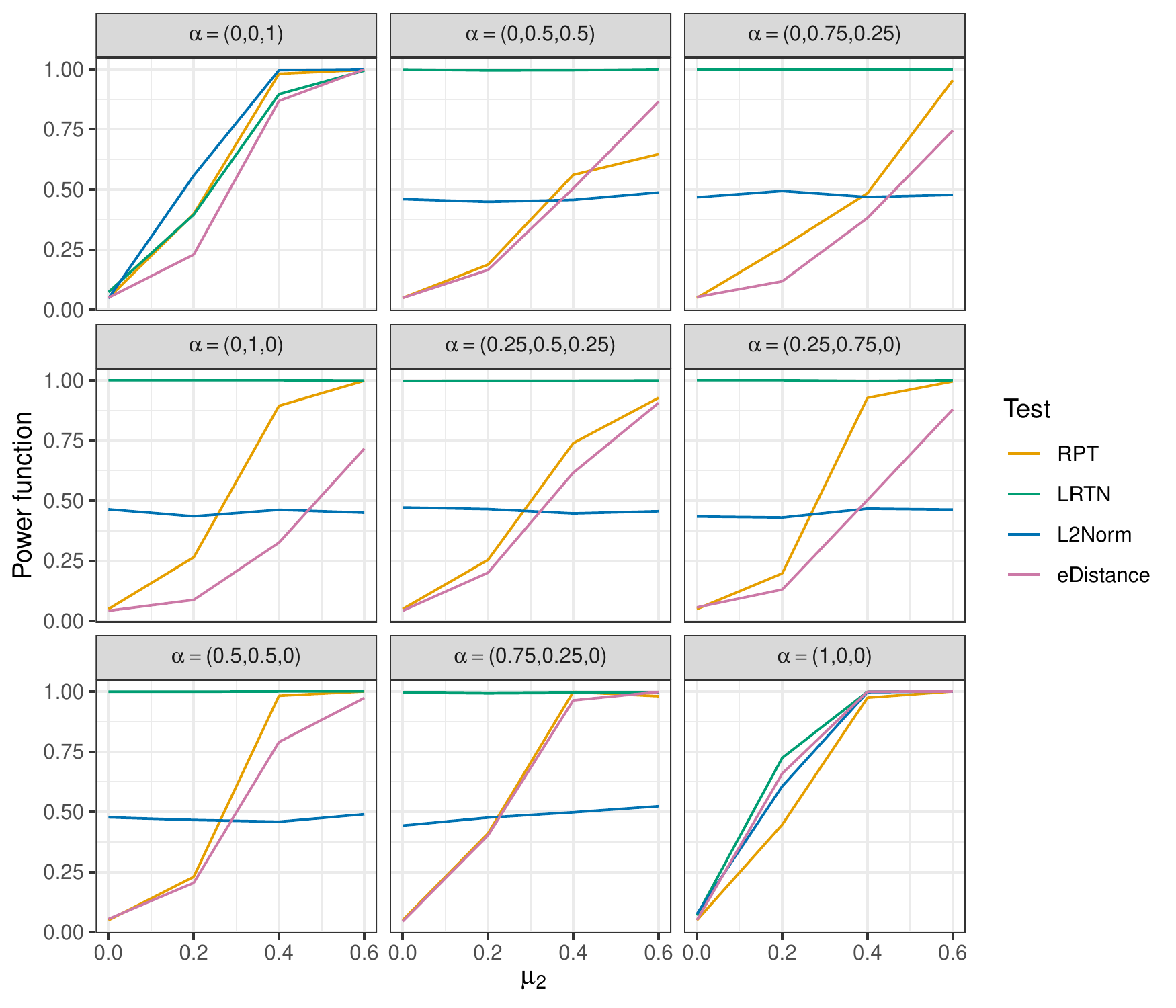}}
\caption{Power functions of the RPT, LRTN, L2Norm, eDistance tests in \S\ref{S:comparison}, for $\mu_2 \in \{0, 0.2,0.4,0.6\}$ and for different values of the vector $\alpha$.}
\label{F:Ex2}
\end{figure}

	
	\section{Application to binary classification}\label{S:binary}
\subsection{Binary classification  of data from elliptical distributions}\label{S:binaryelliptical}	

Linear Discriminant Analysis and Quadratic Discriminant Analysis 
are two of the more classical and well-known learning methods,
both based on the Mahalanobis distance.
They assume that the distributions of the different classes are 
multivariate Gaussian distributions.	
In this section we consider the more challenging learning problem 
where we only assume that the distributions are elliptical. 
	
We propose the following algorithm for binary classification, based on our  results. 
Let $(X_1,Y_1), \ldots, (X_n,Y_n)\in \mathbb R^d \times \{0,1\}$ be an iid  training sample,
where the distributions of $X_1 \vert Y=0$ and $X_1 \vert Y=1$ are unknown elliptical distributions. 
From the training sample, a proportion $\omega \in (0,1)$ will be used to assign weights to  the different directions. Let $n_1$ be the integer part of $(1-\omega)n$, 
and let $\mathcal{X}_{n_1}:= \{(X_1,Y_1), \ldots, (X_{n_1},Y_{n_1})\}\in \mathbb R^d \times \{0,1\}$.
We want to classify the new data $X$. Set $D:=(d^2+d)/2$.

\medbreak
	
	\textbf{Algorithm:}
	\begin{itemize}
		\item Choose an sm-uniqueness set of $D$ directions in $\mathbb R^d$.  
		
	\item	We start using   the data from the subsample $\mathcal{X}_n \setminus \mathcal{X}_{n_1}$ of the training sample  to assign weights $w_j$ to each direction as follows.
		\begin{enumerate}[\normalfont(i)]

		\item For each direction $u_j, ~j\in\{1, \ldots, D\}$, consider the two subsamples
		\begin{equation} \label{project}
		\{	(Z_{1,j},Y_1), \ldots, (Z_{n_1,j},Y_{n_1})\}\in \mathbb R \times \{0,1\},\end{equation}
		\begin{equation} \label{project2}
		\{	(Z_{_{n{_1}+1,j}},Y_{n{_1}+1}), \ldots, (Z_{n,j},Y_{n})\}\in \mathbb R \times \{0,1\},
	\end{equation}
		where $Z_{i,j}=\langle u_j, X_i\rangle,~ i\in\{1, \ldots, n\}$.
		\item Apply the classical $k$-NN rule based on the subsample  (\ref{project}) 
		to classify $Z_{i,j}, i\in\{n_1+1, \ldots,n\}$. That is, let the corresponding score be given by
			
			\begin{align*}
			S_{X_i,j}:= \{
			&\textrm{number of neighbors of the subsample  (\ref{project})}  \\
			&\textrm{with label $1$ among the $k$-nearest data to }  Z_{i,j} \}.
			\end{align*}

\item Set  $p_{X_i,j}:=S_{X_i,j}/k$. If  $p_{X_i,j}> 1/2$ then classify the new data $Z_{i,j}$ as belonging to class $1$ with $i\in\{n_1+1, \ldots,n\}$ (otherwise classify it as  class $0$).  Denote by $\hat Y_{i,j}$ the  label assigned to $Z_{i,j}$.
		
\item 
We define
\[
w_{j}:= \frac{1}{n-n_1} \sum_{i=n_1+1}^{n} \mathcal I_{\hat Y_{i,j}= Y_i}.
\]

\end{enumerate}

\item Lastly, given a new datum $X$, if 
\[
\sum_{j=1}^{D}w_j p_{X,j} \geq  \sum_{j=1}^{D} w_j (1-p_{X,j}),
\]
then classify  $X$ as belonging to class $1$. Otherwise classify it as class~$0$.
\end{itemize}

\begin{remarks} 
(i) One way to choose the set $\mathcal D$ is as follows.  We order the values $w_{j}, \ j\in\{1, \ldots D\}$,
 and take the $\Delta \%$ largest values of $w_{j}$,
  i.e., we  let $\{w^{(1)}, \ldots w^{(D)}\}$ be the order statistics 
  and $R_1, \ldots R_D$ be the corresponding rank statistics. Then   we set
\[
\mathcal D:= \{j : R_j \geq \Delta \times D  \}.
\]
Alternatively, we can initially take more than $D$ directions, and then choose the $D$ best ones. 
			
(ii) Using the weights in the last steps,
we can take into account that some directions may not be very good for classification on the projection,
 see Fig.~\ref{grupos}.
However, we need to split the training sample into two subsamples.
		
\begin{figure}[htb]
\centering
\subfloat{\includegraphics[width=75mm]{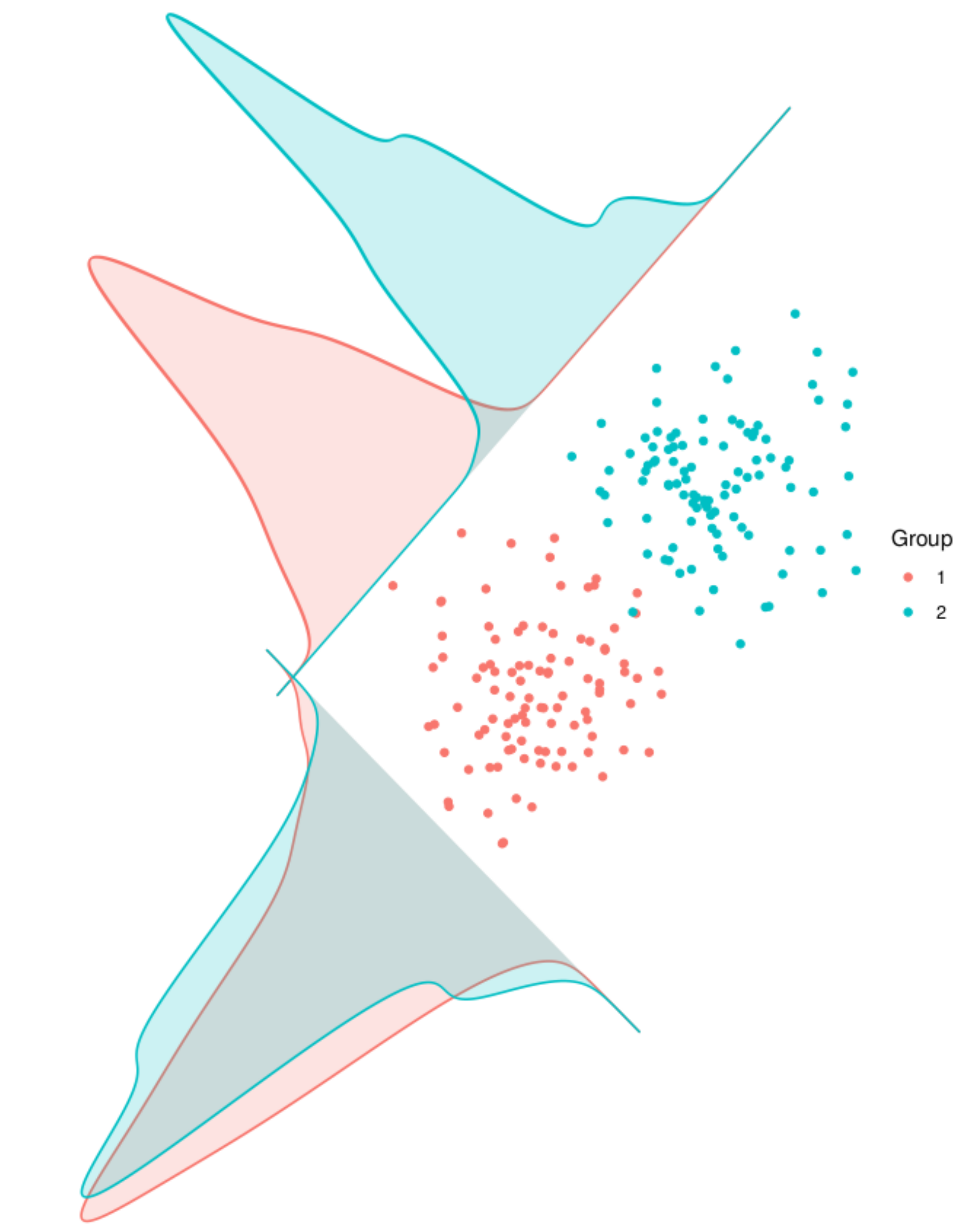}}
\caption{Illustration of Remark~(ii) in \S\ref{S:binaryelliptical}: empirical marginal densities in two random projections of two groups in dimension 2. }
\label{grupos}
\end{figure}	

\end{remarks}

	\subsection{An example with simulated data}\label{S:exsim}
	
To illustrate the performance of the proposed algorithm, we consider a binary classification problem with two classes of size 1000 in dimension 50. The data in both classes are iid and marginally independent with Cauchy distribution. The data in group 1 are centred at the origin. The data in group 2 are shifted in the direction of the vector $\eta=(\eta_1, \ldots, \eta_{50}) \in \mathbb{R}^{50}$, where $\eta_i$ is  chosen uniformly in $[0,3]$ for $i\in\{1, \ldots, 10\}$,  and $\eta_i= 0$ for $i\in\{11, \ldots,50\}$.

We consider $75 \%$ of the data as the training sample and the remaining data as the test sample.  
We repeat  the experiment 100 times under the same conditions.
At each replication, the classifier accuracy is calculated. 

Fig.~\ref{F:rf} shows the accuracy of different classifiers: random forest (RF), support vector machine (SVM), 
generalized linear  model (GLM), and our algorithm using random projections (RP). In the RP algorithm, we take $\Delta= 50 \%$ and $\omega=0.25$.  It is observed that our algorithm performs better than the other competitors.

\begin{figure}[H]
\centering
\subfloat{\includegraphics[width=85mm]{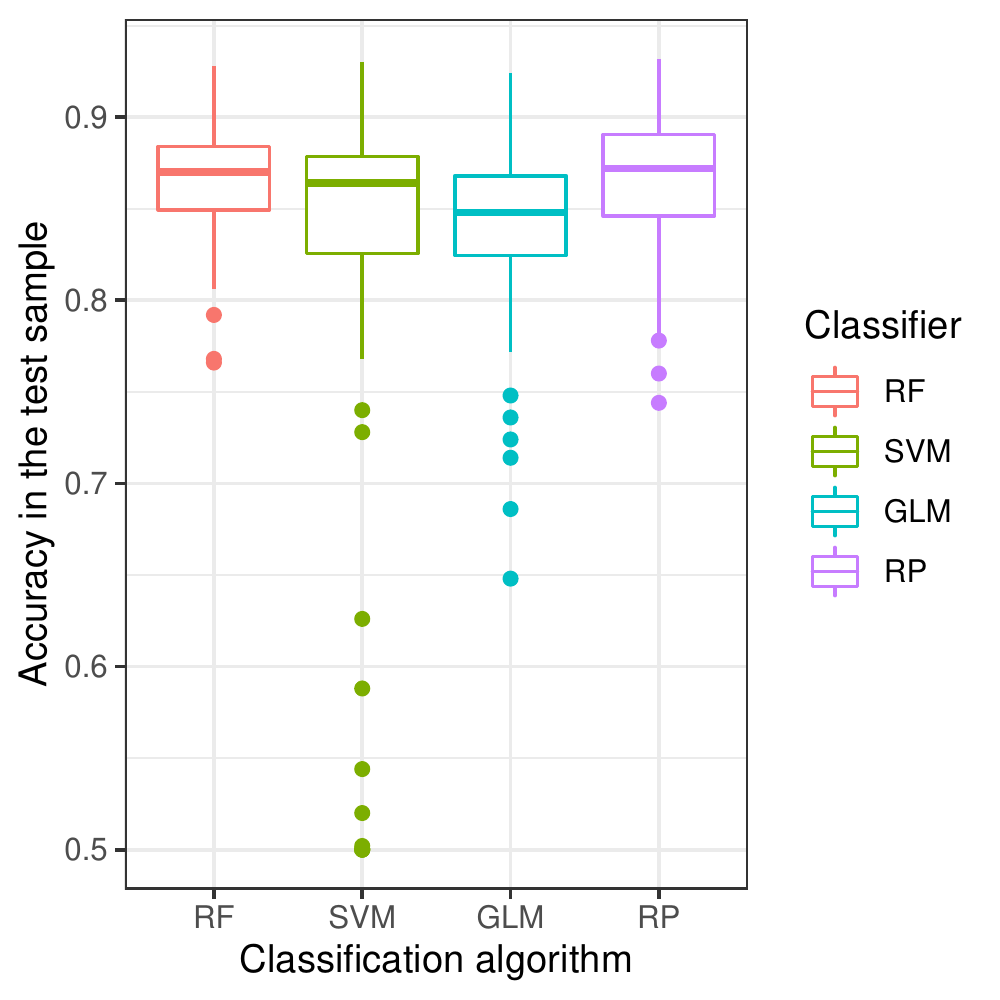}}
\caption{Accuracy in the test sample of the RF, SVM, GLM and RP classifiers in the example in \S\ref{S:exsim}.  }
\label{F:rf}
\end{figure}

\section{Conclusion}\label{S:conclusion}

In this article we have studied the problem of which sets $\mathcal{L}$ of lines in $\mathbb{R}^d$ determine elliptical distributions
in the sense that, if $P,Q$ are elliptical distributions whose projections satisfy $P_L=Q_L$ for all $L\in\mathcal{L}$,
then $P=Q$.
Combining our two main results, Theorems~\ref{T:CWelliptical} and \ref{T:CWellipticalsharp},
we have a precise characterization of such $\mathcal{L}$. 
In particular, there exist such sets $\mathcal{L}$ of cardinality $(d^2+d)/2$, and this number  is best possible.

We have applied our results to derive a Kolmogorov--Smirnov type test for equality of elliptical distributions,
and we have carried out a small simulation study of this test,
comparing it with other tests from the literature.
The results are quite encouraging. 
As a further application of our results, we have proposed
an algorithm for binary classification of data arising
from elliptical distributions, 
again supported by a small simulation.

Let us conclude with a question. 
As mentioned in the introduction, our main result, Theorem~\ref{T:CWelliptical}, is, in a certain sense, a
continuous analogue of  Heppes' result on finitely supported distributions. Recently, the authors 
obtained the following quantitative form of Heppes' theorem, expressed in terms of the total variation metric:
\[
d_{TV}(P,Q):=\sup\{|P(B)-Q(B)|: B \text{~Borel}\}.
\]

\begin{theorem}[\protect{\cite[Theorem~2.1]{FMR23}}]
Let $Q$ be a  probability measure on $\mathbb{R}^d$
whose support contains at most  $k$ points.
Let $H_1,\dots,H_{k+1}$ be  subspaces of $\mathbb{R}^d$ such that 
$H_i^\perp\cap H_j^\perp=\{0\}$ whenever $i\ne j$.
Then, for every Borel probability measure $P$ on $\mathbb{R}^d$, we have
$d_{TV}(P,Q)\le \sum_{j=1}^{k+1} d_{TV}(P_{H_j},Q_{H_j})$.
\end{theorem}	

Is there likewise a quantitative version of Theorem~\ref{T:CWelliptical}?

\section*{Acknowledgments}
Fraiman and Moreno were   supported by grant FCE-1-2019-1-156054, Ag\-encia Nacional de Investigaci\'on e Innovaci\'on, Uruguay. Ransford  was supported by grants from NSERC and the Canada Research Chairs program.

\subsection*{Supplementary material}
All the code in $\textsf{R}$ is available at:

\url{https://github.com/mrleomr/CRAMER-WOLD-ELLIPTICAL}

\bibliographystyle{amsplain}
\bibliography{biblist}

\end{document}